\documentclass[12pt]{amsart}

\newtheorem{theorem}{Theorem}[section]

\newtheorem{lemma}[theorem]{Lemma}
\newtheorem{corollary}[theorem]{Corollary}

\theoremstyle{definition}

\newtheorem{example}[theorem]{Example}

\newtheorem{question}[theorem]{Question}
\newtheorem{conjecture}[theorem]{Conjecture}
\newtheorem{remark}[theorem]{Remark}

\newcommand{\ZZ}{ \ensuremath{\mathbb{Z}}}

\newcommand{\RR}{ \ensuremath{\mathbb{R}}}
\newcommand{\QQ}{ \ensuremath{\mathbb{Q}}}

\newcommand{\skel}{\ensuremath{\mathrm{skel}}}

\newcommand{\gin}{\ensuremath{\mathrm{gin}}}

\newcommand{\init}{\ensuremath{\mathrm{in}}\hspace{1pt}}

\newcommand{\rev}{{\mathrm{{rev}}}}

\newcommand{\Tor}{\ensuremath{\mathrm{Tor}}\hspace{1pt}}
\newcommand{\Hom}{\ensuremath{\mathrm{Hom}}\hspace{1pt}}
\newcommand{\Ext}{\ensuremath{\mathrm{Ext}}\hspace{1pt}}

\newcommand{\lk}{{\mathrm{lk}}}

\newcommand{\mideal}{\ensuremath{\mathfrak{m}}}

\def\cocoa{{\hbox{\rm C\kern-.13em o\kern-.07em C\kern-.13em o\kern-.15em A}}}

\newcommand{\ee}{\mathbf{e}}
\newcommand{\aaa}{\mathbf{a}}
\newcommand{\kk}{\mathbf{k}}

\begin{document}
%\title{On the simplicial polytopes and spheres with fewest k-faces}
\title[On the generalized lower bound conjecture]{On the generalized lower bound conjecture\\ for polytopes and spheres}

\author{Satoshi Murai}
\address{
Satoshi Murai,
Department of Mathematical Science,
Faculty of Science,
Yamaguchi University,
1677-1 Yoshida, Yamaguchi 753-8512, Japan.
}
\email{murai@yamaguchi-u.ac.jp}

\author{Eran Nevo}
\address{
Eran Nevo,
Department of Mathematics,
Ben Gurion University of the Negev,
Be'er Sheva 84105, Israel
}
\email{nevoe@math.bgu.ac.il}

\thanks{
Research of the first author was partially supported by KAKENHI 22740018.
Research of the second author was partially supported by Marie Curie grant IRG-270923 and by ISF grant.
}

\maketitle
\begin{abstract}
In 1971, McMullen and Walkup posed the following conjecture, which is called the
generalized lower bound conjecture:
If $P$ is a simplicial $d$-polytope then its $h$-vector $(h_0,h_1,\dots,h_d)$
satisfies $h_0 \leq h_1 \leq \dots \leq h_{\lfloor \frac d 2 \rfloor}$.
Moreover, if $h_{r-1}=h_r$ for some $r \leq \frac d 2$ then $P$ can be triangulated without introducing simplices of dimension $\leq d-r$.

The first part of the conjecture was solved by Stanley in 1980 using the hard Lefschetz theorem for projective toric varieties.
In this paper, we give a proof of the remaining part of the conjecture.
In addition, we generalize this property to a certain class of simplicial spheres, namely those admitting the weak Lefschetz property.
\end{abstract}

\section{Introduction}\label{sec:Introduction}
%Face enumeration of polytopes has a long history.
The study of face numbers of polytopes is a classical problem.
For a simplicial $d$-polytope $P$ let $f_i(P)$ denote the number of its $i$-dimensional faces, where $-1\leq i\leq d-1$ ($f_{-1}(P)=1$ for the emptyset).
The numbers $f_i(P)$ are conveniently described by the \emph{$h$-numbers}, defined by $h_i(P)=\sum_{j=0}^i (-1)^{j-i} {d-j \choose i-j} f_{j-1}$ for $0\leq i\leq d$. The Dehn-Sommerville relations assert that $h_i(P)=h_{d-i}(P)$ for all $0\leq i\leq \lfloor \frac{d}{2}\rfloor$, generalizing the Euler-Poincar\'{e} formula.

In 1971, McMullen and Walkup \cite{McMullenWalkup:GLBC-71} posed the following \emph{generalized lower bound conjecture} (GLBC), generalizing Barnette's \emph{lower bound theorem} (LBT) \cite{Barnette-LBTfacets-71,Barnette:LBT-73}.

\begin{conjecture}\label{conj:GLBTpolytopes}(McMullen--Walkup)
Let $P$ be a simplicial $d$-polytope. Then
\begin{itemize}
\item[(a)] $1=h_0(P)\leq h_1(P)\leq \dots \leq h_{\lfloor \frac{d}{2} \rfloor}(P)$.
\item[(b)] for an integer $1\leq r\leq \frac{d}{2}$, the following are equivalent:
\begin{itemize}
\item[(i)] $h_{r-1}(P)=h_{r}(P)$.
\item[(ii)] $P$ is $(r-1)$-stacked, namely, there is a triangulation $K$ of $P$ all of whose faces of dimension at most $d-r$ are faces of $P$.
\end{itemize}
\end{itemize}
\end{conjecture}

Around 1980 the \emph{$g$-theorem} was proved, giving a complete characterization of the face numbers of simplicial polytopes. It was conjectured by
McMullen \cite{McMullen-g-conj},
sufficiency of the conditions was proved by Billera--Lee \cite{Billera-Lee} and  necessity by Stanley \cite{Stanley:NumberFacesSimplicialPolytope-80}.
Stanley's result establishes part (a) of the GLBC.

As for part (b), the implication $(\mathrm{ii})\Rightarrow (\mathrm{i})$ was shown in \cite{McMullenWalkup:GLBC-71}. The implication $(\mathrm{i})\Rightarrow (\mathrm{ii})$ is
easy for $r=1$, and was proved for $r=2$ as part of the LBT \cite{Barnette-LBTfacets-71}.
The main goal of this paper is to prove the remaining open part of the GLBC.
In particular, it follows that $(r-1)$-stackedness of
a simplicial $d$-polytope where $r \leq \frac d 2 $ only depends on its face numbers.

McMullen \cite{McMullen:Triangulations} proved that,
to study Conjecture \ref{conj:GLBTpolytopes}(b),
it is enough to consider combinatorial triangulations.
Thus we write a statement in terms of simplicial complexes.
For a simplicial complex $\Delta$ on the vertex set $V$ and a positive integer $i$, let
$$\Delta(i):=\{F\subseteq V: \skel_{i}(2^F)\subseteq \Delta\},$$
where $\skel_{i}(2^F)$ is the $i$-skeleton of the simplex defined by $F$, namely the collection of all subsets of $F$ of size at most $i+1$.

For a simplicial $d$-polytope $P$
with boundary complex $\Delta$,
we say that a simplicial complex $K$ is a \emph{triangulation} of $P$ if its geometric realization is homeomorphic to a $d$-ball and its boundary is $\Delta$.
A triangulation $K$ of $P$ is \emph{geometric} if in addition
there is a geometric realization of $K$ whose underlying space is $P$.

\begin{theorem}\label{thmIntro:GLBCpolytopes}
Let $P$ be a simplicial $d$-polytope with the $h$-vector $(h_0,h_1,\dots,h_d)$,
$\Delta$ its boundary complex, and $1\leq r\leq \frac{d}{2}$ an integer.
If $h_{r-1}=h_{r}$ then $\Delta(d-r)$ is the unique geometric triangulation of $P$
all of whose faces of dimension at most $d-r$ are faces of $P$.
\end{theorem}

Note that the uniqueness of such a triangulation was proved by McMullen \cite{McMullen:Triangulations}.
Moreover, it was shown by Bagchi and Datta \cite{Bagchi-Datta:StellatedSpheres} that if Conjecture \ref{conj:GLBTpolytopes}(b) is true then the triangulation must be $\Delta(d-r)$.

Since the above theorem is described in terms of simplicial complexes,
it would be natural to ask if a similar statement holds for triangulations of spheres, or more generally homology spheres.
Indeed, we also prove an analogous result for homology spheres satisfying a certain algebraic property
called the \emph{weak Lefschetz property} (WLP, to be defined later).

\begin{theorem}\label{thmIntro:GLBCspheres}
Let $\Delta$ be a homology $(d-1)$-sphere having the WLP over a field of characteristic $0$, $(h_0,h_1,\dots,h_d)$ the $h$-vector of $\Delta$, and $1\leq r\leq \frac{d}{2}$ an integer.
If $h_{r-1}=h_{r}$ then $\Delta(d-r)$ is the unique homology $d$-ball with no interior faces of dimension at most $d-r$
and with boundary $\Delta$.
\end{theorem}

Note that an algebraic formulation of the $g$-conjecture (for homology spheres) asserts that any homology sphere has the WLP, see e.g. \cite[Conjecture 4.22]{Swartz-SpheresToManifolds} for a stronger variation.
If this conjecture holds, then Theorem \ref{thmIntro:GLBCspheres} will extend to all homology spheres.
Indeed, the case $r=2$ in Theorem \ref{thmIntro:GLBCspheres} was proved by Kalai \cite{Kalai-LBT}, without the WLP assumption, as part of his generalization of the LBT to homology manifolds and beyond.
Further, note that for $r\leq d/2$, if a homology $(d-1)$-sphere $\Delta$ satisfies that $\Delta(d-r)$ is a homology $d$-ball with boundary $\Delta$, then $\Delta$ satisfies all the numerical conditions in the $g$-conjecture (including the nonlinear Macaulay inequalities), as was shown by Stanley \cite{Stanley:CMcomplexes77}.

This paper is organized as follows:
In Section \ref{sec:Triangulations} we give preliminaries on triangulations and prove the uniqueness claim in the above two theorems.
In Section \ref{sec:CM} we prove that $\Delta(d-r)$ satisfies a nice algebraic property called the Cohen--Macaulay property.
In Section \ref{sec:polytopes}, by using this result together with a geometric and topological argument, we show that $\Delta(d-r)$ triangulates $P$ in Theorem \ref{thmIntro:GLBCpolytopes}.
In Section \ref{sec:LefschetzSpheres} we prove Theorem \ref{thmIntro:GLBCspheres} based on the theory of canonical modules in commutative algebra.
Lastly, in Section \ref{sec:conclude} we give some concluding remarks and open questions.

%%%%%%%%%%%%%%%%%%%%%%%%%%%%%%%%%%%%%%%%%%%%%55
\section{Triangulations}\label{sec:Triangulations}
In this section, we provide some preliminaries and notation on triangulations, and prove the uniqueness statements in Theorems \ref{thmIntro:GLBCpolytopes} and \ref{thmIntro:GLBCspheres}.

Let $\Delta$ be an (abstract) \emph{simplicial complex} on vertex set $V$, namely, a collection of subsets of $V$ such that, for any $F \in \Delta$ and $G \subset F$,
one has $G \in \Delta$.
An element $F \in \Delta$ is called a {\em face} of $\Delta$  and a maximal face (under inclusion) is called a {\em facet} of $\Delta$.
A face $F \in \Delta$ is called an {\em $i$-face} if $\#F=i+1$,
where $\# X$ denotes the cardinality of a finite set $X$.
The {\em dimension} of $\Delta$ is $\dim \Delta=\max\{\# F -1: F \in \Delta \}$.
For $0 \leq k \leq \dim \Delta$,
we write $\mathrm{skel}_k(\Delta)=\{F \in \Delta: \#F \leq k+1\}$ for the {\em $k$-skeleton} of $\Delta$.
Let $f_i=f_i(\Delta)$ be the number of $i$-faces of $\Delta$.
The {\em $h$-vector} $h(\Delta)=(h_0(\Delta),h_1(\Delta),\dots,h_d(\Delta))$ of $\Delta$
is a sequence of integers defined by
$$
h_i(\Delta)=\sum_{j=0}^i (-1)^{j-i} {d-j \choose i-j} f_{j-1}
$$
for $i=0,1,2,\dots,d$, where $d=\dim \Delta+1$ and where $f_{-1}=1$.
If $\Delta$ is the boundary complex of a simplicial polytope $P$,
we also call $h(\Delta)$ the $h$-vector of $P$.

Let $\Delta$ be a simplicial complex on vertex set $V$.
A subset $F \subset V$ is called a \emph{missing face} of $\Delta$ if $F \notin \Delta$ and all proper subsets of $F$ are faces of $\Delta$.
Note that the set of the missing faces of $\Delta$ determines $\Delta$ itself
since it determines all subsets of $V$ which are not in $\Delta$.
It is not hard to see that, by definition, the simplicial complex $\Delta(i)$, defined in the Introduction, is the simplicial complex whose missing faces are the missing faces $F$
of $\Delta$ with $\#F \leq i+1$.
In particular, for $j\leq i$, one has $\Delta(j)=\Delta(i)$ if and only if
$\Delta$ has no missing $k$-faces for $j+1\leq k\leq i$.

The following relation between face numbers and missing faces will be used in the sequel. It was first proved by Kalai \cite[Proposition 3.6]{Kalai:Aspects-94} when $d > 2r+1$,
and was later generalized by Nagel \cite[Corollary 4.8]{Nagel:Empty}.

\begin{lemma}
\label{lem:NoMissingFaces}
Let $\Delta$ be the boundary complex of a simplicial $d$-polytope.
If $h_{r-1}(\Delta)=h_r(\Delta)$ then $\Delta(r-1)=\Delta(d-r)$.
\end{lemma}

\begin{remark}
\label{rem:NoMissingFaces}
Nagel \cite{Nagel:Empty} writes a statement only for simplicial polytopes,
but his proof works for homology spheres admitting  the WLP which we study in Section 5.
\end{remark}

Next, we prove the uniqueness statements in Theorems \ref{thmIntro:GLBCpolytopes} and \ref{thmIntro:GLBCspheres}.
We start with some notations and definitions.
Let $\kk$ be a field.
For a simplicial complex $\Delta$,
let $\widetilde H_i(\Delta;\kk)$ be the $i$th reduced homology group of $\Delta$ with coefficients in $\kk$,
and let
$$\lk_\Delta(F)=\{ G \in \Delta: F \cup G \in \Delta,\ F \cap G = \emptyset\}$$
be the {\em link} of $F$ in $\Delta$.
A $d$-dimensional simplicial complex $\Delta$
is said to be a {\em homology $d$-sphere} (over $\kk$)
if the homology groups
$\widetilde H_{d-\#F-i} (\lk_\Delta(F);\kk)$ are isomorphic to $\kk$ for $i=0$ and vanish for all $i>0$, for all $F \in \Delta$ (including the empty face $\emptyset$).
Also, a {\em homology $d$-ball} (over $\kk$) is
a $d$-dimensional simplicial complex $\Delta$ such that the homology groups $\widetilde H_{d-\#F-i} (\lk_\Delta(F);\kk)$ are either $\kk$ or $0$ for $i=0$ and vanish for $i>0$, for all $F \in \Delta$,
and moreover, its boundary complex
$$\partial \Delta =\{F \in \Delta: \widetilde H_{d-\#F} (\lk_\Delta(F);\kk) =0\}$$
is a homology $(d-1)$-sphere.
We say that a simplicial complex $\Delta$ is a \emph{triangulation} of a topological space $X$ if its geometric realization is homeomorphic to $X$.
Note that a triangulation of a $d$-sphere (resp.\ $d$-ball) is a homology $d$-sphere (resp.\ $d$-ball) over any field.

Let $\Delta$ be a homology $d$-ball.
The faces in $\Delta-\partial \Delta$ are called the \emph{interior faces} of $\Delta$.
If $\Delta$ has no interior $k$-faces for $k \leq d-r$ then $\Delta$ is said to be \emph{$(r-1)$-stacked}.
An {\em $(r-1)$-stacked} sphere (resp.\ homology sphere) is the boundary complex
of an $(r-1)$-stacked triangulation of a ball (resp.\ homology ball).

Recall that a triangulation of a simplicial $d$-polytope $P$ with boundary complex $\Delta$
is a triangulation $K$ of a $d$-ball such that $\partial K=\Delta$.
McMullen \cite[Theorem 3.3]{McMullen:Triangulations} proved that,
for $r \leq \frac d 2$, an $(r-1)$-stacked triangulation $K$ of a simplicial $d$-polytope $P$ is unique.
Moreover, Bagchi and Datta \cite[Corollary 3.6]{Bagchi-Datta:StellatedSpheres}
proved that such a triangulation must be equal to $\Delta(d-r)$
(they actually proved a more general statement for PL-spheres).
We generalize these statements for homology spheres
based on an idea of Dancis \cite{Dancis} who proved that a homology $d$-sphere is determined by its $\lceil \frac{d}{2}\rceil$-skeleton (generalizing a previous work of Perles who showed it for polytopes).
In particular, our result answers \cite[Question 6.4]{Bagchi-Datta:StellatedSpheres}.

\begin{theorem}\label{thm:Uniqueness}
Let $\Delta$ be a homology $(d-1)$-sphere and $1\leq r\leq \frac{d+1}{2}$ an integer.
\begin{itemize}
\item[(i)]
If $\Delta(d-r)$ is a homology $d$-ball with $\partial \Delta(d-r)=\Delta$
then it is $(r-1)$-stacked.
\item[(ii)]
If $\Delta'$ is an $(r-1)$-stacked homology $d$-ball with $\partial \Delta'=\Delta$
then $\Delta'=\Delta(d-r)$.
\end{itemize}
\end{theorem}

\begin{proof}
The statement (i) is obvious since $\Delta(d-r)$ and $\partial \Delta(d-r)=\Delta$ have the same $(d-r)$-skeleton.
We prove (ii).
Since $\Delta'$ is $(r-1)$-stacked, $\Delta'$ has the same $(d-r)$-skeleton as $\Delta$,
and therefore has the same $(d-r)$-skeleton as $\Delta(d-r)$ by definition.
Thus what we must prove is that $\Delta'$ has no missing faces of cardinality $> d-r+1$.
Let $F$ be a $(k+1)$-subset of $[n]$ with $k > d-r$
such that all its proper subsets are in $\Delta'$.
We claim $F \in \Delta'$.

Consider the homology $d$-sphere $S=\Delta'\cup (\{v\}*\Delta)$, where $v$ is a new vertex
and where $\{v\}*\Delta= \Delta \cup \{ \{v\} \cup F: F \in \Delta\}$ is the cone of $\Delta$ with the vertex $v$.
For a subset $W \subset V$, where $V$ is the vertex set of $S$,
let $S|_W=\{G \in S: G \subset W\}$ be the induced subcomplex of $S$ on $W$.
Since all proper subsets of $F$ are in $\Delta'$ and $\Delta'$ is an induced subcomplex of $S$,
to prove $F \in \Delta'$,
it is enough to show that $S|_F$ is not a $(k-1)$-sphere, equivalently that $\widetilde H_{k-1} (S|_F;\kk)=0$.

Since $S-S|_F$ is homotopy equivalent to $S|_{V-F}$ (e.g.\ \cite[Lemma 70.1]{Munkres}),
by Alexander duality (e.g.\ \cite[Theorem 71.1]{Munkres}) and the universal coefficient theorem with field coefficients
we have
$$
\widetilde H_{k-1} (S|_F;\kk) \cong \widetilde H_{d-k}(S-S|_F;\kk) \cong \widetilde H_{d-k}(S|_{V-F};\kk),
$$
so we need to show $\widetilde H_{d-k}(S|_{V-F};\kk)=0$.
Since $d-k \leq r-1 \leq d-r$,
we have $\mathrm{skel}_{d-k}(S|_{V-F})=\mathrm{skel}_{d-k}((\{v\} * \Delta)|_{V-F})$ and $S|_{V-F} \supset (\{v\} * \Delta)|_{V-F}$.
Then, by the definition of the simplicial homology,
we have
$$\dim_\kk \widetilde H_{d-k}(S|_{V-F};\kk) \leq \dim_\kk \widetilde H_{d-k}((\{v\} * \Delta)|_{V-F};\kk).$$
Recall that $v \not \in F$.
The right-hand side of the above inequality is equal to zero
since $(\{v\}*\Delta)|_{V-F}=\{v\}*(\Delta|_{V-F-\{v\}})$ is a cone.
Hence $\widetilde H_{d-k}(S|_{V-F}:\kk)=0$.
\end{proof}

Unlike $(r-1)$-stacked polytopes with $ r \leq \frac d 2$,
$(\frac {d-1} 2)$-stacked simplicial $d$-polytopes cannot be characterized by their $h$-vectors
since $h_{ ({d-1})/ 2}=h_{({d+1})/ 2}$ holds for all simplicial $d$-polytopes when $d$ is odd.
On the other hand, Theorem \ref{thm:Uniqueness} says that
$(\frac {d-1} 2)$-stacked simplicial $d$-polytopes still have a nice combinatorial property.
It would be of interest to have a nice combinatorial characterization of these polytopes.

%%%%%%%%%%%%%%%%%%%%%%%%%%%%%%%%%%%%%%%%%%%%%%%5
\section{Cohen--Macaulayness}\label{sec:CM}

In this section, we prove that the simplicial complexes $\Delta(d-r)=\Delta(r-1)$ in Theorems \ref{thmIntro:GLBCpolytopes} and \ref{thmIntro:GLBCspheres} (the equality holds by Lemma \ref{lem:NoMissingFaces} and Remark \ref{rem:NoMissingFaces} respectively)
satisfy a nice algebraic condition, called the Cohen--Macaulay property.
We first introduce some basic tools in commutative algebra.

\subsection*{Stanley--Reisner rings}
Let $S=\kk[x_1,\dots,x_n]$
be a polynomial ring over an infinite field $\kk$.
For a subset $F \subset [n]=\{1,2,\dots,n\}$,
we write $x_F=\prod _{k \in F} x_k$.
For a simplicial complex $\Delta$ on $[n]$,
the ring
$$\kk[\Delta]=S/I_\Delta$$
where $I_\Delta=(x_F: F \subset [n],\ F \not \in \Delta)$,
is called the {\em Stanley--Reisner ring} of $\Delta$.

The simplicial complex $\Delta(i)$ has a simple expression in terms of Stanley--Reisner rings.
For a homogeneous ideal $I \subset S$, let $I_{\leq k}$ be the ideal generated by all elements in $I$ of degree $ \leq k$.
Since the missing faces of $\Delta$ correspond to the minimal generators of $I_\Delta$
and since $\Delta(i)$ is the simplicial complex whose missing faces are the missing faces $F$ of $\Delta$
with $\#F \leq i+1$,
one has
$$I_{\Delta(i)}=(I_\Delta)_{\leq i+1}.$$

\subsection*{Cohen--Macaulay property}
Let $I \subset S$ be a homogeneous ideal and $R=S/I$.
The Krull dimension $\dim R$ of $R$ is the minimal number $k$
such that there is a sequence of linear forms $\theta_1,\dots,\theta_k \in S$
such that $\dim_\kk S/(I+(\theta_1,\dots,\theta_k))< \infty$.
If $d=\dim R$, then a sequence $\Theta=\theta_1,\dots,\theta_d$ of linear forms such that
$\dim_\kk S/(I+(\Theta))< \infty$ is called a {\em linear system of parameters} of $R$ (l.s.o.p.\ for short).
A sequence of homogeneous polynomials $f_1,\dots,f_r$ of positive degrees is called
a {\em regular sequence} of $R$ if $f_i$ is a non-zero divisor of $S/(I+(f_1,\dots,f_{i-1}))$
for all $i=1,2,\dots,r$.
We say that $R$ is {\em Cohen--Macaulay} if every l.s.o.p.\ of $R$ is a regular sequence of $R$.

A simplicial complex $\Delta$ is said to be {\em Cohen--Macaulay} (over $\kk$)
if $\kk[\Delta]$ is a Cohen--Macaulay ring.
The following topological criterion for the Cohen--Macaulay property was proved by Reisner \cite{Reisner}.

\begin{lemma}[Reisner's criterion]\label{lem:Reisner}
A simplicial complex $\Delta$ is Cohen--Macaulay (over $\kk$) if and only if,
for any face $F \in \Delta$,
$\widetilde H_i(\lk_\Delta(F); \kk)=0$ for all $i \ne \dim \lk_\Delta(F)$.
\end{lemma}

\subsection*{The weak Lefschetz property}
Let $I \subset S$ be a homogeneous ideal such that $R=S/I$ has dimension $0$.
We write $R=\bigoplus_{i=0}^s R_i$, where $R_i$ is the homogeneous component of $R$ of degree $i$ and where $R_s \ne 0$.
We say that $R$ has the {\em weak Lefschetz property} (WLP for short) if there is a linear form $w \in R_1$,
called a {\em Lefschetz element} of $R$,
such that the multiplication $\times w :R_k \to R_{k+1}$
is either injective or surjective for all $k$.

We say that a ring $R=S/I$ of dimension $d>0$,
where $I$ is a homogeneous ideal, has the WLP if it is Cohen--Macaulay
and there is an l.s.o.p.\ $\Theta$	of $R$ such that $S/(I+(\Theta))$
has the WLP.
Also, a simplicial complex $\Delta$ is said to have the WLP (over $\kk$)
if $\kk[\Delta]$ has the WLP.
It is known that the boundary complex of a simplicial polytope has the WLP
over $\QQ$.
See \cite[Section 5.2]{Fulton}.

It is known that, for any homogeneous ideal $I \subset S$,
the Hilbert series $H(S/I,t)=\sum_{i=0}^\infty (\dim_\kk (S/I)_i) t^i$ of the ring $S/I$ can be written in the form
$$H(S/I,t)= \frac {h_0+h_1t + \dots + h_s t^s} {(1-t)^d}$$
where $d=\dim S/I$ and where $h_s \ne 0$.
See \cite[Corollay 4.1.8]{Bruns-Herzog}.
The vector $h(S/I)=(h_0,h_1,\dots,h_s)$ is called the $h$-vector of $S/I$.
If $S/I$ has the WLP then its $h$-vector is unimodal,
namely it satisfies $h_0 \leq \dots \leq h_p \geq h_{p+1} \geq \dots \geq h_s$
for some $p$.
Indeed, let $R=S/(I+(\Theta))$, where $\Theta$ is an l.s.o.p.\ of $S/I$.
Then $h_k=\dim_\kk R_k$ for all $k$.
If $R$ has the WLP and $h_p \geq h_{p+1}$ for some $p$,
then the multiplication $w : R_{k} \to R_{k+1}$ is surjective
for $k=p$. The multiplication map is also surjective for all $k \geq p$ as $S$ is generated by elements of degree $1$ and $p+1\geq 1$,
and we have $h_p \geq \dots \geq h_s$.

\subsection*{Generic initial ideals}
Here we briefly recall generic initial ideals.
We do not give details on this subject.
\cite{Gr} and \cite[Section 4]{HibiHerzog} are good surveys on generic initial ideals.

Let $>_\rev$ be the degree reverse lexicographic order induced by the ordering $x_1 >_\rev  \dots >_\rev x_n$.
For a homogeneous ideal $I \subset S$,
let $\init_{>_\rev}(I)$ be the initial ideal of $I$ w.r.t.\ $>_\rev$.
Let $\mathrm{GL}_n(\kk)$ be the general linear group with coefficients in $\kk$.
Any $\varphi=(a_{ij}) \in \mathrm{GL}_n(\kk)$ induces an automorphism of $S$,
again denoted by $\varphi$,
$$\varphi\big(f(x_1,\dots,x_n)\big)=
f\left(\sum_{k=1}^n a_{k1} x_k,\dots,\sum_{k=1}^n a_{kn} x_k\right)$$
for any $f \in S$.
It was proved by Galligo that $\init_{>_\rev}(\varphi(I))$ is constant for a generic choice of $\varphi \in \mathrm{GL}_n(\kk)$.
See \cite[Theorem 1.27]{Gr}.
This monomial ideal $\init_{>_\rev}(\varphi(I))$ is called the {\em generic initial ideal} of $I$ w.r.t.\ $>_{\rev}$,
and denoted $\gin(I)$.
We need the following well-known property on the WLP.

\begin{lemma}
\label{gin}
Let $I \subset S$ be a homogeneous ideal and $d=\dim S/I$.
\begin{enumerate}
\item[(i)] $S/I$ is Cohen--Macaulay if and only if $S/\gin(I)$ is Cohen--Macaulay.
\item[(ii)] $S/I$ has the WLP if and only if $S/\gin(I)$ has the WLP.
Moreover, if $S/I$ has the WLP, then
$x_n,\dots,x_{n-d+1}$ is an l.s.o.p.\ of $S/\gin(I)$
and
$x_{n-d}$ is a Lefschetz element of $S/(\gin(I)+(x_n,\dots,x_{n-d+1}))$.
\end{enumerate}
\end{lemma}

See \cite[Corollary 4.3.18]{HibiHerzog} for the first statement.
The second statement follows from \cite[Lemmas 4.3.7]{HibiHerzog}
together with the facts that, for $\theta_1,\dots,\theta_{d+1} \in S$ generic linear forms,
$\theta_1,\dots,\theta_d$ is an l.s.o.p.\ of $S/I$ and $\theta_{d+1}$ is a Lefschetz element of $S/(I+(\theta_1,\dots,\theta_d))$,
and that for a generic choice of $\varphi \in \mathrm{GL}_n(K)$ the linear forms
$x_n,\dots,x_{n-d}$ are generic for $S/\varphi(I)$.

The following result due to Mark Green \cite[Proposition 2.28]{Gr} is crucial to prove the Cohen--Macaulay property of $\Delta(r-1)$.

\begin{lemma}[Crystallization Principle]
\label{crysterization}
Suppose $\mathrm{char}(\kk)=0$.
Let $I \subset S$ be a homogeneous ideal generated by elements of degree $\leq m$.
If $\gin(I)$ has no minimal generators of degree $m+1$
then $\gin(I)$ is generated by elements of degree $\leq m$.
%and $\reg(S/I) \leq m-1$.
\end{lemma}

\begin{theorem}\label{thm:CM}
Suppose $\mathrm{char}(\kk)=0$.
Let $I \subset S$ be a homogeneous ideal such that $S/I$ has the WLP,
and let $h(S/I)=(h_0,h_1,\dots,h_s)$.
If $h_{r-1} = h_r = h_{r+1}$ for some $ 1 \leq r \leq s-1$,
then $S/I_{\leq r}$ is Cohen--Macaulay of dimension $\dim S/I +1$.
\end{theorem}

\begin{proof}
Let $J= \gin(I)$ and $d=\dim S/I$.
We first claim that $S/J_{\leq r}$ is Cohen--Macaulay.
Observe that $J$ is a monomial ideal.
By Lemma \ref{gin}, $S/J$ is Cohen--Macaulay of dimension $d$,
and $J$ has no minimal generators which are divisible by one of $x_n,\dots,x_{n-d+1}$.
Also, since $ h_{r-1}=h_r=h_{r+1}$, the WLP shows that the multiplication
\begin{align}
\label{wlp}
\times x_{n-d} : S/\big(J+(x_n,\dots,x_{n-d+1})\big)_j \to
S/\big(J+(x_n,\dots,x_{n-d+1})\big)_{j+1}
\end{align}
is injective for $ j \leq r$,
which implies that $J$ has no minimal generators of degree $\leq r+1$ which are divisible by $x_{n-d}$.
Indeed, if there is a minimal generator of the form $ux_{n-d}$, then $u$ is in the kernel of the map \eqref{wlp}.
Thus $J_{\leq r}$ has no minimal generators which are divisible by
one of $x_n,\dots,x_{n-d}$.
Thus $x_n,\dots,x_{n-d}$ is a regular sequence of $S/J_{\leq r}$.
In particular, we have $\dim S/J_{\leq r} \geq d+1$
since the length of a regular sequence is bounded by the dimension (\cite[Proposition 1.2.12]{Bruns-Herzog}).

It is left to show that the quotient by this regular sequence is a finite dimensional vector space over $\kk$.
Since
the multiplication map \eqref{wlp}
is surjective when $j=r-1$, $(S/J+(x_n,\dots,x_{n-d}))_{r}=0$
and $J$ contains all monomials in $\kk[x_1,\dots,x_{n-d-1}]$ of degree $r$.
Thus $\dim_\kk S/(J_{\leq r}+(x_n,\dots,x_{n-d})) < \infty$, and $S/J_{\leq r}$ is Cohen--Macaulay of dimension $d+1$ with an l.s.o.p.\ $x_n,\dots, x_{n-d}$.

Next, we prove $\gin(I_{\leq r})=\gin(I)_{\leq r}$.
By the Crystallization principle, what we must prove is that $\gin(I_{\leq r})$ has no minimal generators of degree $r+1$.
Since $I_{\leq r} \subset I$ and $(I_{\leq r})_{r}=I_{r}$,
it is enough to prove that $\gin(I)$ has no minimal generators of degree $r+1$.
Indeed, we already showed that $J=\gin(I)$ has no minimal generator of degree $r+1$ which
is divisible by one of $x_n,\dots,x_{n-d+1},x_{n-d}$.
We also showed that $J$ contains all monomials in $\kk[x_1,\dots,x_{n-d-1}]$
of degree $r$.
These facts guarantee that $J=\gin(I)$ has no minimal generators of degree $r+1$, as desired.

We proved that $S/\gin(I_{\leq r})=S/\gin(I)_{\leq r}$ is Cohen--Macaulay of dimension $d+1$.
Then the desired statement follows from Lemma \ref{gin}(i).
\end{proof}

\begin{corollary}\label{cor:CM}
Suppose $\mathrm{char}(\kk)=0$.
Let $\Delta$ be a homology $(d-1)$-sphere having the WLP over $\kk$.
If $h_{r-1}(\Delta)=h_r(\Delta)$ for some $r \leq \frac d 2$,
then $\Delta(r-1)$ is Cohen--Macaulay over $\kk$ and has dimension $d$.
\end{corollary}

\begin{proof}
Recall that
the $h$-vector of $\Delta$ coincides with the $h$-vector
of its Stanley--Reisner ring $\kk[\Delta]$.
Since the $h$-vector of $\Delta$ is symmetric,
the WLP shows $h_{r-1}(\Delta)=h_r(\Delta)=\dots=h_{d-r+1}(\Delta)$.
Since $I_{\Delta(r-1)}=(I_\Delta) _{\leq r}$,
Theorem \ref{thm:CM}
says that $\kk[\Delta(r-1)]$ is Cohen-Macaulay of dimension $d+1$.
Thus $\Delta(r-1)$ is Cohen--Macaulay of dimension $d$.
\end{proof}

\begin{remark}\label{rem:vanKampen}
The weaker assertion that $\dim(\Delta(d-r))\leq d$ for $r\leq \frac{d}{2}$ is true for any simplicial $(d-1)$-sphere $\Delta$, and more generally for any simplicial complex $\Delta$ which embeds in the $(d-1)$-sphere.

This can be shown using van-Kampen obstruction to embedability, see \cite{VanKampen, Shapiro, Wu}, for cones over Flores complexes \cite{Flores}.
If we assume $\dim(\Delta(d-r))>d$ then, for $d$ even, $\Delta$ contains
$\skel_{d/2}(2^{[d+2]})$, hence it contains the cone over Flores complex $L=\skel_{\frac{d}{2}-1}(2^{[d+1]})$. (Here $[i]:=\{1,2,\dots,i\}$.)
By the non-vanishing on $L$ of the van-Kampen obstruction to embedability in the $(d-2)$-sphere, we conclude that the cone over $L$ does not embed in the $(d-1)$-sphere, a contradiction.
The argument for $d$ odd is similar.
\end{remark}
%%%%%%%%%%%%%%%%%%%%%%%%%%%%%%%%%%%%%%%%%%%%%%
\section{GLBC for polytopes}\label{sec:polytopes}
In this Section we prove the existence part of Theorem \ref{thmIntro:GLBCpolytopes}.

\begin{theorem}\label{thm:GLBCpolytopes}
Let $P$ be a simplicial $d$-polytope with the $h$-vector $(h_0,h_1,\dots,h_d)$,
$\Delta$ its boundary complex, and $1\leq r\leq \frac{d}{2}$ an integer.
If $h_{r-1}=h_{r}$ then $\Delta(d-r)$ is a geometric triangulation of $P$.
\end{theorem}

In the rest of this section,
we fix a simplicial $d$-polytope $P$ satisfying the assumption of Theorem \ref{thm:GLBCpolytopes},
and prove the theorem for $P$.

We may assume $P \subset \RR^d$.
Let $V=\{v_1,v_2,\dots,v_n\} \subset \RR^d$ be the vertex set of $P$
and let $\Delta$ be the boundary complex of $P$.
For a subset $T =\{v_{i_1},\dots,v_{i_k}\} \subset V$,
we write $[T]=\mathrm{conv}(v_{i_1},\dots,v_{i_k})$ for the convex hull of the vertices in $T$.
Let $\Delta'=\Delta(r-1)$.

\begin{lemma}\label{lem:D'embedded}
The set $\{[F]: F \in \Delta'\}$ is a geometric realization of $\Delta'$,
namely,
\begin{itemize}
\item[(i)] $[F_1]\cap[F_2]=[F_1\cap F_2]$ for all $F_1,F_2 \in \Delta'$, and
\item[(ii)] $\dim [F]=\# F -1$ for all $F \in \Delta'$.
\end{itemize}
\end{lemma}

\begin{proof}
The proof is similar to that of  \cite[Proposition 3.4]{Bagchi-Datta:StellatedSpheres}.

(i)
Assume by contradiction that $F_1,F_2 \in \Delta'$ form a counterexample to (i) with the size $\# F_1 + \# F_2$ minimal.
Note that the convex set $[F_1]\cap [F_2]$ is not contained in the boundary of $P$, as otherwise it would equal a single face $[F]$ with $F\in \Delta$ and thus $F_1\cap F_2=F$, which says that (i) holds for $F_1$ and $F_2$.
In particular,
we have $F_1 \not \in \Delta$ and $F_2 \not \in \Delta$.
We prove the following properties for $F_1$ and $F_2$.
\begin{itemize}
\item[(a)] Any $p \in [F_1] \cap [F_2] \setminus [F_1 \cap F_2]$ is in the relative interior of both $[F_1]$ and $[F_2]$.
\item[(b)] $F_1\cap F_2=\emptyset.$
\item[(c)] $[F_1]$ and $[F_2]$ intersect in a single point.
\end{itemize}

We first prove (a).
Suppose to the contrary that $p$ is in the boundary of $[F_1]$.
Then there is an $u \in F_1$ such that $p \in [F_1 -\{u\}]$.
Since $p \not \in [F_1 \cap F_2]$,
we have $p \in [F_1 -\{u\}] \cap [F_2] \setminus [(F_1- \{u\}) \cap F_2]$,
contradicting the minimality of $F_1$ and $F_2$.
Hence (a) holds.

Next we show (b).
Let $p \in [F_1] \cap [F_2] \setminus [F_1 \cap F_2]$.
By (a), there are convex combinations with positive coefficients
$\sum_{v\in F_1}a_v v=p=\sum_{v\in F_2}b_v v$ with $\# F_1 \geq 2$ and $\#F_2\geq 2$.
If there is $u\in F_1\cap F_2$, say with $a_u\leq b_u$, then by subtracting $a_u u$ from both sides and by normalizing them
we get a point $q$ which is contained in $[F_1-\{u\}] \cap [F_2]$.
Since $q$ is in the relative interior of $[F_1-\{u\}]$ by the construction,
we have $q \notin [(F_1- \{u\}) \cap F_2]$, contradicting the minimality.
Hence (b) holds.

We finally prove (c).
Suppose to the contrary that $[F_1] \cap [F_2]$ contains two different points $p$ and $q$.
Let $\ell$ be the line through them.
Then the endpoints of the line segment $\ell \cap [F_1]\cap[F_2]$ must be on the boundary of either $[F_1]$ or $[F_2]$,
contradicting (a) as $[F_1 \cap F_2]$ is empty by (b).
Hence (c) holds.

We now complete the proof of (i).
By (a) and (c), the intersection of $[F_1]$ and $[F_2]$ equals the intersection of their affine hulls,
as otherwise the neighborhood of $p$ in $[F_1]\cap[F_2]$ is not a single point.
This fact and (b) say $\# F_1+ \#F_2 \leq d+2$.
However,
since $F_1$ and $F_2$ are not in $\Delta$
and since $\Delta'=\Delta(d-r)$ and $\Delta$
have the same $(d-r)$-skeleton,
we have $\#F_1 \geq d-r+2$ and $\#F_2 \geq d-r+2$, a contradiction.
Hence we conclude that (i) holds.

(ii)
Lemma \ref{lem:NoMissingFaces} and Theorem \ref{thm:CM} show that $\Delta'$ is $d$-dimensional and pure, namely all of its facets have cardinality $d+1$.
Thus it is enough to show that if $F=\{v_{i_1},\dots,v_{i_{d+1}}\}$ is a facet of $\Delta'$ then $\dim [F]=d$. Suppose to the contrary that $\dim [F]<d$.
Then $v_{i_1},\dots,v_{i_{d+1}}$ are in the same hyperplane in $\RR^d$.
Thus by Radon's theorem there is a partition $F=F'\cup F''$ such that $[F']\cap [F'']\neq \emptyset$. This contradicts (i).
\end{proof}

Let $[\Delta']=\cup_{F \in \Delta'} [F]$ be the underlying space of the geometric simplicial complex $\{[F]:F \in \Delta'\}$.
To complete the proof of Theorem \ref{thm:GLBCpolytopes}, it is left to show
\begin{lemma}\label{lem:D'equalsP}
$[\Delta']=P$.
\end{lemma}
\begin{proof}
Observe that $[\Delta']\subseteq P$.
Assume by contradiction that there is a point $p\in P-[\Delta']$.
We assume that $[\Delta']$ and $P$ are embedded in $S^d$
via the natural homeomorphism $\RR^d \cong S^d - v \subset S^d$,
where $v$ is a point in $S^d$.
% Let $S$ be the union of $P$ with a (topological) cone $v*[\Delta]$ where $v\notin P$, so $S$ is a $d$-sphere.
Let $q \in \RR^d - P$.
Since $[\Delta']$ contains the boundary of $P$,
$p$ and $q$ are in different connected components in $S^d-[\Delta']$.
Thus $S^d-[\Delta']$ is not connected.
By Alexander duality,
we have $\widetilde H_{d-1}([\Delta'];\QQ) \cong \widetilde H_0(S^d-[\Delta']; \QQ) \ne 0$.

Recall that $\Delta$ has the WLP over $\QQ$.
Thus $\Delta'$ is Cohen--Macaulay over $\QQ$ of dimension $d$ by Corollary \ref{cor:CM}. By Lemma \ref{lem:D'embedded} $[\Delta']$ is the underlying space of a geometric realization of $\Delta'$.
Thus Reisner's criterion (Lemma \ref{lem:Reisner}) says $\widetilde H_{d-1}([\Delta'];\QQ)=0$, a contradiction.
\end{proof}
%%%%%%%%%%%%%%%%%%%%%%%%%%%%%%%%%%%%%%%

\section{GLBC for Lefschetz spheres}\label{sec:LefschetzSpheres}
In this section we prove the existence part in Theorem \ref{thmIntro:GLBCspheres}.
The proof is algebraic
and we assume familiarity with $\ZZ^n$-graded commutative algebra theory.
See e.g.\ \cite{MillerSturmfels} for the basics of this theory.

First, we set some notation.
Let $\ee_i \in \ZZ^n$ be the $i$th unit vector of $\ZZ^n$.
We consider the $\ZZ^n$-grading of $S=\kk[x_1,...,x_n]$ defined by $\deg x_i = \ee_i$.
For a $\ZZ^n$-graded $S$-module $M$ and for $\aaa=(a_1,\dots,a_n) \in \ZZ^n$,
we denote by $M_\aaa$ the graded component of $M$ of degree $\aaa \in \ZZ^n$.
Let $\mideal=(x_1,\dots,x_n)$ be the graded maximal ideal of $S$.
We regard $\kk$ as a graded $S$-module by identification $\kk=S/\mideal$.
We recall a few known properties on $\Tor_i^S(\kk,-)$.

\begin{lemma}
\label{koszul}
Let $C$ be a graded $S$-module.
If $C_k =0$ for all $k \leq r$ then one has
$\Tor_{i}(\kk,C)_{i+j}= 0$ for all $i$ and $j \leq r$.
\end{lemma}

\begin{proof}
Let $\mathcal K_\bullet= \mathcal K_\bullet (x_1,\dots,x_n)$ be the Koszul complex of $x_1,\dots,x_n$ (see, e.g., \cite[\S 1.6]{Bruns-Herzog}).
Since $\mathcal K_\bullet$ is the minimal free  resolution of $\kk$,
$$\Tor_i(\kk,C)_{i+j} \cong H_i(\mathcal K_\bullet \otimes C)_{i+j}.$$
On the other hand, all the elements in $\mathcal K_i$ have degree $\geq i$
and all the elements in $C$ have degree $\geq r+1$ by the assumption.
These facts imply that
$(\mathcal K_i \otimes C)_{i+j}=0$ for $j \leq r$.
Hence $H_i(\mathcal K_\bullet \otimes C)_{i+j}=0$ for all $j \leq r$.
\end{proof}

The following fact on generic initial ideals is well-known.
See \cite[Theorem 2.27]{Gr}.
%\eran{before, this Lemma was part of the Crystallization Principal. Do you use a different reference now?}
%\satoshi{I changed this since I don't want to introduce Tor in section 3.
%Note that the statement for regularity in Crystallization principal is a consequence of the lemma below.}

\begin{lemma}[Bayer--Stillman]
\label{reg}
Suppose $\mathrm{char}(\kk)=0$.
Let $I \subset S$ be a homogeneous ideal.
If $\gin(I)$ is generated by monomials of degree $\leq m$
then
$\Tor_i^S(\kk,S/I)_{i+j}=0$ for all $j \geq m$.
\end{lemma}

We also recall some basic facts on canonical modules.
For a subset $F \subset [n]$,
let $\ee_F= \sum_{i \in F} \ee_i$.
For a Cohen--Macaulay $\ZZ^n$-graded ring $R=S/I$ of dimension $d$,
the module $\omega_R=\Ext_S^{n-d}(R,S(-\ee_{[n]}))$ is called the {\em canonical module} of $R$.
An important property of a canonical module is that
it is isomorphic to the Matlis dual of the local cohomology module $H_\mideal^d(R)$
by the local duality
(see \cite[Corollay 3.5.9]{Bruns-Herzog}).
Now suppose that $R=\kk[\Delta]$.
Then the local duality and the Hochster's formula for local cohomology \cite[Theorem 5.3.8]{Bruns-Herzog} imply that, for any $F \in \Delta$, one has
\begin{align}
\label{3}
\dim_\kk (\omega_{\kk[\Delta]})_{\ee_F}
= \dim_\kk (H_\mideal^d(\kk[\Delta]))_{-\ee_F}
=\dim_\kk \widetilde H_{d-1-\# F}(\mathrm{lk}_\Delta(F)).
\end{align}
Recall that by Reisner's criterion homology balls and spheres are Cohen--Macaulay.

The next result and Theorem \ref{thm:Uniqueness} prove Theorem \ref{thmIntro:GLBCspheres}.

\begin{theorem}\label{thm:GLBCLefschetzSpheres}
Suppose $\mathrm{char}(\kk)=0$.
Let $\Delta$ be a homology $(d-1)$-sphere having the WLP.
If $h_{r-1}(\Delta)=h_r(\Delta)$ for some $r \leq \frac d 2$
then $\Delta(r-1)$ is a homology $d$-ball whose boundary complex is $\Delta$.
\end{theorem}

\begin{proof}
\textit{Step 1:}
Let $\Delta'=\Delta(r-1)$ and $C=I_{\Delta}/I_{\Delta'}$.
For a graded $S$-module $M$, let $\mathrm{ann}_S(M)=\{g \in S: gf = 0 \mbox{ for all }f \in M\}.$
We first show that $C$ satisfies the following conditions:
\begin{enumerate}
\item[(i)] $\mathrm{ann}_S(C)=I_{\Delta'}$.
\item[(ii)] $C$ is Cohen--Macaulay of dimension $d+1$.
\item[(iii)] $\Tor_{n-d-1}^S (\kk,C)\cong \kk(-\ee_{[n]})$.
\end{enumerate}

(i)
$\mathrm{ann}_S(C) \supset I_{\Delta'}$ is clear.
It is enough to show that there is an element $f \in I_{\Delta}$
such that $ g f \not \in I_{\Delta'}$ for all $g \in S$ with $g \not \in I_{\Delta'}$.
Let $F_1,\dots,F_s$ be the facets of $\Delta'$.
By Corollary \ref{cor:CM} each $F_i$ is of size $d+1$.
We claim that the polynomial $f= \sum_{i=1}^s x_{F_i} \in I_\Delta$ satisfies the desired property.

To prove this,
since $C$ contains $\bigoplus_{i=1}^s x_{F_i} \cdot(S/(x_k: k \not \in F_i))$ as a submodule,
it is enough to show that, for any $g \not \in I_{\Delta'}$, $g x_{F_i} \ne 0$ in $x_{F_i} \cdot (S/(x_k: k \not \in F_i))$ for some $i$.
Moreover,
since $x_{F_i} \cdot (S/(x_k: k \not \in F_i))$ is $\ZZ^n$-graded,
we may assume that $g=x_{i_1}^{a_1} \cdots x_{i_t}^{a_t}$, where $a_1,\dots,a_t$ are not zero.
Since $x_{i_1}^{a_1} \cdots x_{i_t}^{a_t} \not \in I_{\Delta'}$, we have $\{i_1,\dots,i_t\} \in \Delta'$.
Then since $\Delta'$ is Cohen--Macaulay, $\Delta'$ is pure and there is a facet $F_i$ which contains $\{i_1,\dots,i_t\}$.
Then we have $x_{i_1}^{a_1} \cdots x_{i_t}^{a_t}x_{F_i} \ne 0$ in $x_{F_i}\cdot (S/(x_k: k \not \in F_i))$ as desired.

(ii)
Consider the
short exact sequence
\begin{align}
\label{2}
0 \longrightarrow
C\longrightarrow
S/I_{\Delta'}
\longrightarrow
S/I_{\Delta}
\longrightarrow
0.
\end{align}
Since $S/I_{\Delta'}$ is Cohen--Macaulay of dimension $d+1$
and since $S/I_{\Delta}$ is Cohen--Macaulay of dimension $d$,
we conclude that $C$ is Cohen--Macaulay of dimension $d+1$ (e.g.\ use the depth lemma \cite[Proposition 1.2.9]{Bruns-Herzog}).

(iii)
It remains to prove
$\Tor_{n-d-1}^S(\kk,C)  \cong \kk(-\ee_{[n]})$.
Note that $\Tor_{n-d}^S(\kk,S/I_{\Delta'})=0$
since $S/I_{\Delta'}$ is Cohen--Macaulay of dimension $d+1$.
Then the short exact sequence \eqref{2} induces the exact sequence
$$
0
\longrightarrow
 \Tor_{n-d}^S(\kk,S/I_{\Delta})_j
\longrightarrow
 \Tor_{n-d-1}^S(\kk,C)_j
\longrightarrow
 \Tor_{n-d-1}^S(\kk,S/I_{\Delta'})_j
\longrightarrow \cdots
$$
for all $j \geq 0$.
Since $\Delta$ is a homology $(d-1)$-sphere,
$\Tor_{n-d}^S(\kk,S/I_{\Delta})\cong \kk(-\ee_{[n]})$.
On the other hand,
since $\gin(I_{\Delta'})$ has no generators of degrees $\geq r+1$
as we showed in the proof of Theorem \ref{thm:CM},
we have
$\Tor_{n-d-1}^S(\kk,S/I_{\Delta'})_j=0$ for $j \geq n-d-1+r$
by Lemma \ref{reg}.
These facts and the exact sequence imply
$$\bigoplus_{j \geq n-d-1+r} \Tor_{n-d-1}^S(\kk,C)_j\cong \kk(-\ee_{[n]}).$$
On the other hand, since $I_{\Delta'}=(I_{\Delta})_{\leq r}$, we have $C_k = 0$ for $k \leq r$.
This implies $\Tor_{n-d-1}^S(\kk,C)_j=0$ for $j < n-d-1+r$ by Lemma \ref{koszul}, and (iii) follows.

\textit{Step 2:}
We show that $C \cong \Ext_S^{n-d-1}(\kk[\Delta'],S(-\ee_{[n]}))=\omega_{\kk[\Delta']}$.
It is standard in commutative algebra that conditions (i), (ii) and (iii) imply this isomorphism, but we include its proof.
Since $C$ is Cohen--Macaulay of dimension $d+1$,
it follows from \cite[Theorem 3.3.10]{Bruns-Herzog}
that
\begin{align}
\label{2-1}
\Ext_S^{n-d-1}\big(\Ext_S^{n-d-1}\big(C,S(-\ee_{[n]})\big),S(-\ee_{[n]})\big)=C.
\end{align}
On the other hand, by the duality on resolutions of $C$ and $\Ext_S^{n-d-1}(C,S(-\ee_{[n]}))$,
we have
$$\Tor_0^{S}(\kk,\Ext_S^{n-d-1}(C,S(-\ee_{[n]})))_{\aaa} \cong \Tor_{n-d-1}^S(\kk,C)_{\ee_{[n]}-\aaa}$$
for all $\aaa \in \ZZ^n$
(see \cite[Corollary 3.3.9]{Bruns-Herzog}).
Then the condition (iii) of Step 1 implies that $\Ext_S^{n-d-1}(C,S(-\ee_{[n]}))$ has a single generator in degree $0$,
so $\Ext_S^{n-d-1}(C,S(-\ee_{[n]}))\cong S/J$ for some ideal $J$.

We claim that $J=I_{\Delta'}$, equivalently $\mathrm{ann}_S(\Ext_S^{n-d-1}(C,S(-\ee_{[n]})))=I_{\Delta'}$.
Since $\mathrm{ann}_S(M) \subset \mathrm{ann}_S(\Hom_S(M,N))$ for all $S$-modules $M$ and $N$,
\eqref{2-1} says
$$
\mathrm{ann}_S(C) \subset \mathrm{ann}_S(\Ext_S^{n-d-1}(C,S(-\ee_{[n]})))
\subset \mathrm{ann}_S(C),
$$
which implies $\mathrm{ann}_S(\Ext_S^{n-d-1}(C,S(-\ee_{[n]})))=\mathrm{ann}_S(C)=I_{\Delta'}$ by (i) of Step 1.
Now $C \cong \Ext_S^{n-d-1}(\kk[\Delta'],S(-\ee_{[n]}))$ follows from \eqref{2-1}
since $\Ext_S^{n-d-1}(C,S(-\ee_{[n]})) \cong S/I_{\Delta'}=\kk[\Delta']$.

\textit{Step 3:}
We now prove the theorem.
By the Hochster's formula \eqref{3},
for any $F \in \Delta'$
we have
\begin{align*}
\dim_\kk \widetilde H_{d -\#F} \big(\mathrm{lk}_{\Delta'}(F)\big)=
\dim_\kk (\omega_{\kk[\Delta']})_{\ee_F}=
\dim_\kk (I_{\Delta}/I_{\Delta'})_{\ee_F}
=
\begin{cases}
1, & \mbox{ if } F \not \in \Delta, \\
0, & \mbox{ otherwise.}
\end{cases}
\end{align*}
Clearly the above equation
together with $\Delta'$ being Cohen--Macaulay
imply that $\Delta'$ is a homology ball whose boundary complex is equal to $\Delta$.
\end{proof}

The proof given in this section is quite algebraic.
It would be of interest to have a combinatorial or a topological proof of Theorem \ref{thm:GLBCLefschetzSpheres}.
%%%%%%%%%%%%%%%%%%%%%%%%%%%%%%%%%%%%%%%%%5
%%%%%%%%%%%%%%%%%%%%%%%%%%%%%%%%%%%%%%%%%%%%%

\section{Concluding Remarks}\label{sec:conclude}
It is easy to see that ($1$-)stacked spheres are boundaries of stacked polytopes,
and that their stacked triangulations are shellable.
Then it is natural to ask %that the following condition holds.
\begin{question}
\label{6-1}
Let $\Delta$ be an $(r-1)$-stacked $d$-ball with $r \leq \frac{d+1}{2}$.
Then
\begin{itemize}
\item[(i)] is it true that $\partial \Delta$ is polytopal?
\item[(ii)] is it true that $\Delta$ is shellable?
\end{itemize}
\end{question}

The next examples show that the answers to the above questions
are negative.

\begin{example}
\label{e1}
Let $B$ be Rudin's non-shellable triangulation of a $3$-ball \cite{Rudin}. Its $f$-vector is $(1,14,66,94,41)$ and its $h$-vector is $(1,10,30,0,0)$.
Let $K$ be the join of $B$ and a simplex $\sigma$ of dimension $k \geq 2$.
Then $K$ is a $(k+4)$-ball.
Also, the interior faces of $K$ are exactly those containing both $\sigma$ and an interior face of $B$.
Then, since $B$ contains no interior vertices,
$K$ is $2$-stacked.

On the other hand, $K$ is not shellable since $B$ is not shellable. Indeed, a shelling order on $K$ would induce a shelling order on $B$ by deleting $\sigma$ from all facets in the shelling order of $K$.

Also, $\partial K$ is non-polytopal. Indeed, assume the contrary, then for $v$ a vertex of $\sigma$, $\lk_{\partial K}(\sigma -\{v\})=B \cup (\{v\} * \partial B)$  is the boundary complex of a polytope. Thus, there is a Bruggesser--Mani \emph{line shelling} of $\lk_{\partial K}(\sigma -\{v\})$ which adds the facets with $v$ last (see \cite[Section 8.2]{Ziegler} for details),
so first it shells $B$, a contradiction.
\end{example}

\begin{example}
\label{e2}
%Next, we note that
There exists a large number of shellable $(r-1)$-stacked $d$-balls with $r \leq \frac d 2$
whose boundary is non-polytopal. Indeed, fixing $d$,
Goodman and Pollack \cite{GoodmanPollack:fewPolytopes-86} showed that the log of the number of combinatorial types of boundaries of simplicial $d$-polytopes on $n$ vertices is at most $O(n \log(n))$.
On the other hand,
the log of the number of Kalai's squeezed $(d-1)$-spheres satisfying $h_{r-1}=h_r$, where $r \leq \frac d 2$,
is at least $\Omega(n^{r-2})$
(see \cite{Kalai-manyspheres} for the details).
Since Kalai's squeezed spheres satisfying $h_{r-1}=h_r$ are known to be the boundaries of $(r-1)$-stacked shellable balls (see \cite{Kalai-manyspheres} and \cite{Kleinschmidt-Lee84} for details),
they give a large number of $(r-1)$-stacked triangulations of a $d$-ball whose boundary is non-polytopal when $r \geq 4$.
\end{example}

Although the answers to Question \ref{6-1} are negative in general,
it would be of interest to study these problems for special cases.
Below, we write a few open questions on stacked balls and spheres.

\begin{conjecture}\label{Q:shellablePolytopes}
Let $P$ be an $(r-1)$-stacked $d$-polytope with $r \leq \frac{d+1}{2}$.
\begin{itemize}
\item[(i)] (McMullen \cite{McMullen:Triangulations})
The $(r-1)$-stacked triangulation of $P$ is {\em regular}.
\item[(ii)] (Bagchi--Datta \cite{Bagchi-Datta:StellatedSpheres})
The $(r-1)$-stacked triangulation of $P$ is shellable.
\end{itemize}
\end{conjecture}
Note that Conjecture \ref{Q:shellablePolytopes}(i)
implies Conjecture \ref{Q:shellablePolytopes}(ii).
McMullen's original conjecture considered the case $r \leq \frac d 2$,
but we want to include the case $r= \frac {d+1} 2$ in view of Theorem \ref{thm:Uniqueness}.
Also, it would be of interest to study the geometric meaning of the triangulation given in Theorem \ref{thmIntro:GLBCpolytopes}.

We see that there exists a non-shellable $2$-stacked ball whose boundary is non-polytopal in Example \ref{e1},
and that there even exists a shellable $3$-stacked balls whose boundary is non-polytopal in Example \ref{e2}.
But the following question is open.

\begin{question}\label{Q:polytopal}
Let $\Delta$ be a $2$-stacked triangulation of a $d$-ball which is shellable.
Is $\partial \Delta$ polytopal?
\end{question}

Finally, we raise the following question concerning Theorem \ref{thmIntro:GLBCspheres}.

\begin{question}
With the same notation as in Theorem \ref{thmIntro:GLBCspheres},
is it true that if $\Delta$ is a triangulation of a sphere then $\Delta(d-r)$ is a triangulation of a ball?
\end{question}

It seems to be plausible that if $\Delta$ is a PL-sphere then $\Delta(d-r)$ is a PL-ball.
But we do not have an answer even for this case.
\section*{Acknowledgments}
We would like to thank Gil Kalai and Isabella Novik for helpful comments on an earlier version of this paper.

\bibliography{gbiblio(modified)}
\bibliographystyle{plain}

\end{document}